\documentclass[rqno,12pt]{amsart}

\usepackage{amsthm}
\usepackage{amsmath}
\usepackage{amssymb}
\usepackage{eucal}
\usepackage{amsfonts}
\usepackage{amsmath}
\usepackage[Lenny]{fncychap}
\usepackage{enumerate}
\usepackage{amssymb}
\usepackage[english]{babel}
\usepackage[latin1]{inputenc}
\usepackage[breaklinks=true]{hyperref}
\usepackage[all]{xy}
\usepackage{fancybox}
\usepackage{latexsym,mathrsfs,longtable,lscape,dsfont,yfonts}
\usepackage{srcltx}

\ifx\pdfpagewidth\undefined % bessere Schrift bei PDF
\usepackage[T1]{fontenc}
\else % = bessere Schrift bei PDF
\usepackage[OT1]{fontenc} % = bessere Schrift bei PDF
\fi % = bessere Schrift bei PDF
\usepackage{amsfonts,amssymb,latexsym,longtable,amsmath}

\newtheorem{thm}{Theorem}[section]
\newcommand{\bthm}{\begin{thm}} \newcommand{\ethm}{\end{thm}}
\newtheorem{prop}[thm]{Proposition}
\newcommand{\bprp}{\begin{prop}} \newcommand{\eprp}{\end{prop}}
\newtheorem{fact}[thm]{Fact}
\newcommand{\bfct}{\begin{fact}} \newcommand{\efct}{\end{fact}}
\newtheorem{prob}[thm]{Problem}
\newcommand{\bprb}{\begin{prob}} \newcommand{\eprb}{\end{prob}}
\newtheorem{quest}[thm]{Question}
\newcommand{\bqtn}{\begin{quest}} \newcommand{\eqtn}{\end{quest}}
\newtheorem{lem}[thm]{Lemma}
\newcommand{\blem}{\begin{lem}} \newcommand{\elem}{\end{lem}}
\newtheorem{claim}[thm]{Claim}
\newcommand{\bclm}{\begin{claim}} \newcommand{\eclm}{\end{claim}}
\newtheorem{cor}[thm]{Corollary}
\newcommand{\bcor}{\begin{cor}} \newcommand{\ecor}{\end{cor}}
\newtheorem{conj}[thm]{Conjecture}
\newcommand{\bcnj}{\begin{conj}} \newcommand{\ecnj}{\end{conj}}
\theoremstyle{definition}
\newtheorem{defn}[thm]{Definition}
\newcommand{\bdfn}{\begin{defn}} \newcommand{\edfn}{\end{defn}}
\newtheorem{spec}[thm]{Specializing}
\newcommand{\bspc}{\begin{spec}} \newcommand{\espc}{\end{spec}}
\theoremstyle{remark}
\newtheorem{rem}[thm]{Remark}
\newcommand{\brem}{\begin{rem}} \newcommand{\erem}{\end{rem}}
\newtheorem{cnv}[thm]{Convention}
\newcommand{\bcnv}{\begin{cnv}} \newcommand{\ecnv}{\end{cnv}}
\newtheorem{exam}[thm]{Example}
\newcommand{\bexm}{\begin{exam}} \newcommand{\eexm}{\end{exam}}
\newcommand{\bpf}{\begin{proof}} \newcommand{\epf}{\end{proof}}

\newtheorem{thmy}{\textbf{Theorem}}
\newenvironment{thmx}{\stepcounter{thm}\begin{thmy}}{\end{thmy}}

\newcommand{\Z}{\mathbb Z}

\newcommand{\N}{\mathbb N}

\renewcommand{\phi}{\varphi}
\renewcommand{\theta}{\vartheta}

\newcommand{\ga}{{\alpha}}

\newcommand{\gl}{{\lambda}}

\newcommand{\gb}{{\beta}}
\newcommand{\gga}{{\gamma}}
\newcommand{\gw}{{\omega}}
\newcommand{\gs}{{\sigma}}
\newcommand{\vv}{\mathbf{v}}

\newcommand{\supp}{{\rm supp}}

\newcommand{\nbd}{neighborhood \ }

\newcommand{\skp}{\smallskip}
\newcommand{\mkp}{\medskip}
\newcommand{\bkp}{\bigskip}

\begin{document}

\title[Homomorphic encoders of profinite abelian groups~I]{Homomorphic encoders of profinite abelian groups~I}

\author[M. Ferrer]{Mar\'ia V. Ferrer}
\address{Universitat Jaume I, Instituto de Matem\'aticas de Castell\'on,
Campus de Riu Sec, 12071 Castell\'{o}n, Spain.}
\email{mferrer@mat.uji.es}

\author[S. Hern\'andez]{Salvador Hern\'andez}
\address{Universitat Jaume I, INIT and Departamento de Matem\'{a}ticas,
Campus de Riu Sec, 12071 Castell\'{o}n, Spain.}
\email{hernande@mat.uji.es}

\thanks{ Research Partially supported by the Spanish Ministerio de Econom\'{i}a y Competitividad,
grant: MTM/PID2019-106529GB-I00 (AEI/FEDER, EU) and by the Universitat Jaume I, grant UJI-B2019-06}
%\vspace{1cm}

\begin{abstract}
Let $\{G_i :i\in\N\}$ be a family of finite Abelian groups. We say that a subgroup  $G\leq \prod\limits_{i\in \N}G_i$ is
\emph{order controllable} if for every $i\in \mathbb{N}$ there is $n_i\in \mathbb{N}$ such that for each $c\in G$, there exists
$c_1\in G$ satisfying that $c_{1|[1,i]}=c_{|[1,i]}$, $supp (c_1)\subseteq [1,n_i]$, and order$(c_1)$
divides order$(c_{|[1,n_i]})$. In this paper we investigate the structure of order controllable subgroups. It is
proved that every order controllable, profinite, abelian group contains a subset
$\{g_n : n\in\N\}$ that topologically generates the group and whose elements $g_n$ all have finite support.
As a consequence, sufficient conditions are obtained that allow us to encode, by means of a topological group isomorphism,
order controllable profinite abelian groups. Some applications of these results to group codes will appear subsequently \cite{FH:2021}.
%Looking at the specific case of group codes, we obtain that if $G$ is an order controllable, shift invariant, group code over a finite abelian group $H$,
%then $G$ possesses a canonical generating set. Furthermore, our construction also yields that $G$ is algebraically conjugate to a full group shift.
%Some connections to coding theory are also highlighted.
\end{abstract}

%\footnote{Revise the definition of order controllable}

%\dedicatory{}

\date{24 October, 2019}

\maketitle \setlength{\baselineskip}{24pt}

\section {Introduction}
Let $\Z$ and $\N$ respectively denote the group of integers and the semigroup of natural numbers.
Suppose that $\Z$ is given the discrete topology and $\Z^\N$ the corresponding product
topology. Nunke proved in \cite{nunke} that every infinite, closed, subgroup $G$ of $\Z^\N$
is topologically isomorphic to a product of infinite cyclic groups, i.e., the group $G$ contains a subset
$\{g_n : n\in\N\}$ such that $G\cong \prod\limits_{n\in\N}{\langle g_n\rangle }$. Furthermore, it is not hard to
prove that the elements $g_n$ can be selected with finite support if and only if
$G\cap \Z^{(\N)}$ is dense in $G$ (here $\Z^{(\N)}$ denotes the direct sum, that is, the subgroup of the product
consisting of all elements with finite support). In this case, we say that $\{g_n : n\in\N\}$ is a \emph{generating set} that encodes the
group $G$. The main goal of this paper is to study the existence of generating sets, in a profinite abelian group $G$,
whose elements have finite support. %We base our approach on the notion of \emph{order controllable} group.
We prove that for order controllable groups it is always possible to find a generating set whose elements all have finite support.
As a consequence, we can build some topological group isomorphisms that encode an order controllable closed subgroup of a product
of finite abelian group and describes how the subgroup is placed within the product of finite groups where it is located. 
Some applications of these results to group codes will appear subsequently \cite{FH:2021}.
\mkp

Let $\{G_i:i\in I\}$ be a family of topological groups. As usual, its {\em direct product\/}
$\prod_{i\in I} G_i$ is the set of all functions
$g: I\to \bigcup\{G_i:i\in I\}$ such that $g(i)\in G_i$ for every $i\in I$.
The group operation on $\prod_{i\in I} G_i$ is defined coordinate-wise: the product
$gh\in \prod_{i\in I} G_i$ of $g$ and $h$ in $\prod_{i\in I} G_i$ is the function defined by
$gh(i)=g(i)h(i)$ for each $i\in I$.
Clearly,
the identity element $e$ of $\prod_{i\in I} G_i$ is the function that assigns the identity element $e_i$ of $G_i$ to every $i\in I$.
We equip this product with the canonical product topology.
The subgroup
$$
\bigoplus_{i\in I} G_i=\{g\in \prod_{i\in I} G_i: g(i)=e_i
\mbox{ for all but finitely many } i\in I\}
$$
%of $\prod_{i\in I} G_i$
is called the {\em direct sum\/} of the family $\{G_i:i\in I\}$.
The \emph{support} of an element $x\in \prod_{i\in I} G_i$ is the set
$$\supp(x):=\{ i\in I : x_i\not= e_i\}.$$

Given a subgroup  $G\leq \prod\limits_{i\in I}G_i$ and a subset
$J\subseteq I$, we denote by $G_{J}:=\{c\in G :c(j)=e_j,j\notin J\}$ and $G_{|J}:=\pi_J(G)$,
where $\pi_J:\prod\limits_{i\in I}G_i\rightarrow \prod\limits_{i\in J}G_i$ is the the canonical projection. %onto $\prod\limits_{i\in J}G_i$.
When $G_i=M$ for every $i\in I$, then we write $M^I$ instead of $\prod_{i\in I} G_i$.

If $S$ is a subset of a group $G$, then we denote by $\langle S\rangle$ the subgroup generated by $S$, that is, the smallest subgroup of $G$ containing every element of $S$.
However, the symbol $\langle g \rangle$ will denote the cyclic subgroup generated by $\{g\}$, $g\in G$.
Since most results here concern abelian groups, we will use additive notation from here on.
In particular, we will denote the identity element by $0$.

%For $J\subseteq I$, the projection
%$\pi_{IJ}: \prod_{i\in I} G_i\to \prod_{j\in J} G_j$
%defined by $\pi_{IJ}(g)=g_{|J}$ for $g\in G$, is the canonical group
%homomorphism. When the index set $I$ is clear from the context, we use $\pi_{|J}$ instead of $\pi_{|IJ}$ for brevity.
\medskip

The following two group-theoretic notions that have stem in coding theory.
%characterize the way a group is embedded into a direct product of groups.

\bdfn\label{def 1}
A subgroup  $G\leq \prod\limits_{i\in I}G_i$ is called \emph{weakly controllable} if
$G\cap \bigoplus\limits_{i\in I} G_i$ is dense in $G$, that is, if $G$ is generated by
its elements with finite support. The group $G$ is called \emph{weakly observable} if
$G\cap \bigoplus\limits_{i\in I} G_i=\overline G\cap \bigoplus\limits_{i\in I} G_i$,
where $\overline G$ stands for the closure of $G$ in $\prod\limits_{i\in I}G_i$ for
the product topology.
\edfn
\mkp

Although the notion of (weak) controllability was coined by Fagnani earlier
in a broader context (cf. \cite{fagnani:adv97,fagnani_zampieri:ieee96}), both notions
were introduced in the area of coding theory by Forney and Trott (cf. \cite{forney_trott:04}).
They observed that if
the groups $G_i$ are locally compact abelian, then controllability and observability are dual properties
with respect to the Pontryagin duality: If $G$ is a closed subgroup of $\prod\limits_{i\in I}G_i$,
then it is weakly controllable if and only if its annihilator
$G^\bot=\{\chi\in \widehat{\prod\limits_{i\in I}G_i} : \chi(G)=\{0\}\}$ is a weakly observable subgroup
of $\bigoplus\limits_{i\in I}\widehat{G_i}\leq \prod\limits_{i\in I}\widehat{G_i}$ (cf. \cite[4.8]{forney_trott:04}).

We now describe different ways in which a subgroup is placed in a product of topological
groups.

\bdfn\label{def 2}
Let $\{G_i\}_{i\in \N}$ be a family of compact groups and let $G$ be a closed subgroup of the product $\prod\limits_{i\in \N}G_i$.
The subgroup $G$ is called \emph{rectangular} if there is a subgroup $H_i\leq G_i$ for all $i\in \N$ such that
$G=\prod_{i\in \N}{H_i}$.
We say that $G$ is \emph{topologically generated} by the set $\{g_n : n\in \N\}$ if all elements $g_n,\ n\in\N$, have finite support
and the subgroup $\bigoplus\limits_{n\in\N}\, \langle g_n\rangle $ is dense in $G$.
If, in addition, the map
$$\Phi\colon \bigoplus\limits_{n\in\N}{\langle g_n\rangle }\to G$$
defined by $$\Phi((x_n)):=\sum\limits_{n\in\N}x_n,$$
with $x_n\in\langle g_n \rangle$ for all $n\in\N$, extends to a topological (onto) group isomorphism
$$\Phi\colon \prod\limits_{n\in\N}{\langle g_n\rangle }\to G$$
we say that $G$ is \emph{is weakly rectangular} and $\Phi$ is an \emph{isomorphic encoder} of $G$.
Finally, if $$\Phi(\bigoplus\limits_{n\in\N}\, \langle g_n\rangle )= G\cap\bigoplus\limits_{i\in \N} G_i,$$
we say that $G$ is an \emph{implicit direct product}.
\edfn

\medskip

\noindent The observations below are easily verified. (cf. \cite{GL:2012}).

\begin{enumerate}%[(a)]
\item{Weakly rectangular subgroups and rectangular subgroups of $\prod\limits_{i\in \N}G_i$ are weakly controllable.}
\item{If each $G_i$ is a pro-$p_i$-group for some prime $p_i$, and all $p_i$ are distinct, then every closed
subgroup of the product $\prod\limits_{i\in \N}G_i$ is rectangular, and thus is an implicit direct product.}
\item{If each $G_i$ is a finite simple non-abelian group, %and all $G_i$ are distinct,
then every closed normal subgroup of the
product  $\prod\limits_{i\in \N}G_i$ is rectangular, and thus an implicit direct product.}
\end{enumerate}
\medskip

The main goal addressed in this paper is to investigate when a profinite abelian group is weakly rectangular or
an implicit direct product of finite groups. In particular we aim to know to what extent the converse of (1) above holds.
That is, we are interested in the following question (cf. \cite{GL:2012}):

\begin{prob}\label{Question}
Let $\{G_i : i\in \N \}$ be a family of finite abelian groups, and $G$
a closed subgroup of the product $\prod\limits_{i\in \N}G_i$. If $G$ is weakly controllable, that is
$G\cap \bigoplus\limits_{i\in \N} G_i$ is dense in $G$, what
can be said about the structure of $G$? More precisely, under what additional conditions on the group
$G$  there exists a generating set $\{y_j : j\in L \}$ for $G$?. In particular,
when is $G$ weakly rectangular or an implicit direct product? %of the family $\{G_i : i\in \N \}$?

%that is, weakly rectangular and
%$G\cap\bigoplus\limits_{i\in \N} G_i=\bigoplus\limits_{i\in\mathbb{N}}H_i$, where each $H_i$ is a subgroup
%of $\bigoplus\limits_{j\in F_i} G_j$ for some $F_i\subseteq \mathbb{N}$?
\end{prob}
\medskip

A first step in order to tackle this  question, was given in \cite{FHS:2017,kiehlmann}, where the following result was established.

\begin{thm}\label{DHS:products}
Let $I$ be a countable set, $\{G_i:i\in \N\}$ be a family of finite abelian groups
and $\prod_{i\in \N} G_i$ be its direct product.
If $G$ is a closed weakly controllable subgroup of $\prod_{i\in \N} G_i$, then
$G$ is topologically isomorphic to a direct product of finite cyclic groups.
\end{thm}
\medskip

We notice that this result does not answer Problem \ref{Question}, since its actual proof gives no clue about the existence
of generating set for $G$.
%Thus, Question \ref{Question}, remains still open so far.
Incidently, the continuity of mappings defined on weak direct sums has been investigated in \cite{DikShaSpe:16,Spe:18}.
However, those results go in a different direction and Question \ref{Question} is not addressed there.

\brem
The relevance of these notions stem from coding theory where they appear
in connection with the study of (convolutional) group codes \cite{forney_trott:04,rosenthal}.
However, similar concepts had been studied in symbolic dynamics previously. Thus, the notions
of weak controllability and weak observability are related to the concepts of \emph{irreducible shift} and
\emph{shift of finite type}, respectively, that appear in symbolic dynamics.
Here, we are concerned with abelian profinite groups and
our main interest is to clarify the overall topological and algebraic structure of abelian profinite groups
that satisfy any of the properties introduced above.
In the last section, we shall also highlight some connections with the study of group codes.
\erem
\mkp

We now formulate our main result. Here, for every group $G$, we denote by $(G)_p$ the largest $p$-subgroup of $G$ and
$\mathbb{P}_G=\{p\in\mathbb{P}\,:\, G\hbox{ contains a }p-\hbox{subgroup}\}$
where %$p\in \mathbb{P}_G$ and
$\mathbb{P}$ is the set of all prime numbers.

\begin{thmx}\label{theorem_A}
%\begin{thm}\label{th_products} (t$h_{}$products)
\emph{Let $G$ be an order controllable, closed, subgroup of a countable product $\prod\limits_{i\in\mathbb{N}}G_i$ of finite abelian groups $G_i$.
Then the following assertions hold true:
\begin{enumerate}[(a)]
\item There is a generating set $\{y_{m}^{(p)}\,:\,m\in\mathbb{N},p\in\mathbb{P}_G\}\subseteq G\cap(\bigoplus\limits_{i\in\mathbb{N}}G_i)$
for $G$ such that $\{y_m^{(p)}\,:\,m\in\mathbb{N}\}\subseteq(G\cap(\bigoplus\limits_{i\in\mathbb{N}}G_i))_p$ for all prime $p$.
\item If $G$ has finite exponent, then there is an isomorphic encoder
$$\Phi\colon \prod\limits_{\begin{array}{l}_{m\in\mathbb{N}}\\^{p\in\mathbb{P}_G}\end{array}}\langle y_m^{(p)}\rangle\to G$$
%then $G$ is topologically isomorphic to $\prod\limits_{\begin{array}{l}_{m\in\mathbb{N}}\\^{p\in\mathbb{P}_G}\end{array}}\langle y_m^{(p)}\rangle$
and, as a consequence, $G$ is weakly rectagular.
\item If $\bigoplus\limits_{m\in\mathbb{N}} \langle y_m^{(p)}\rangle[p]$ is weakly observable for each prime $p$, %\in \mathbb{P}_G$,
then $G$ is an implicit direct product.
\end{enumerate}
}
\end{thmx}
%\mkp

\section{Basic definitions and terminology}

In accordance with the general terminology, a group $G$ is called {\em torsion or periodic} if
the orders of all its elements are finite, {\em torsion-free} if all elements, except the identity,
have infinite order. If there is a natural number $n$ such that $ng=0$ for all $g\in G$, %(the identity element of the group),
we say that $G$ has {\it finite exponent}.  Then the smallest such $n$ is called the
{\it exponent} of $G$, denoted as exp$(G)$.
An abelian torsion group $G$ in which the order of every element is a power of
a prime number $p$ is called {\em $p$-group}.
An element $g$ of a $p$-group $G$ is said to have \emph{ finite height $h:=h(g,G)$ in $G$}
if this is the largest natural number $n$ such that the equation $p^nx = g$ has a solution $x\in G$.
We say that $g$ has \emph{infinite height} if  the solution exists  for all $n\in \mathbb{N}$.
Here on, the symbol $G[p]$ denotes the subgroup consisting of all
elements of order $p$. It is well known that $G[p]$ is a vector space on the field $\mathbb{Z}(p)$.
%Let $G=\bigoplus G_i$ be a countable sum of groups and let $x$ be an element of $G$. In the sequel, $i_f(x)$
%(resp. $i_l(x)$) designates the first index (resp. last index) $i$ such that $x_i\not=0$.

%In the sequel, if $I$ is a set, the symbol $\mathcal P_f(I)$ will denote the family of all finite subsets of $I$.

\begin{defn}\label{def_2}  Let $\{G_i : i\in \mathbb{N} \}$ be a family of topological groups and $G$ a subgroup of $\prod\limits_{i\in \mathbb{N}} G_i$. We have the
following notions:
\begin{enumerate}[(1)]
%\item $G$ is \emph{weakly controllable} if $G\bigcap\bigoplus\limits_{i\in \mathbb{N}}G_i$ is dense in $G$.

\item $G$ is \emph{controllable} if for every %$c\in G$ and
$i\in \mathbb{N}$ there is $n_i\in \mathbb{N}$ such that for each $c\in G$, there exists
$c_1\in G$ such that $c_{1|[1,i]}=c_{|[1,i]}$ and $c_{1|]n_i,+\infty[}=0$
(we assume that $n_i$ is the least natural number satisfying this property).
Remark that this property implies the existence of $c_2:=c-c_1\in G$ such that $c=c_1+c_2$,
$supp (c_1)\subseteq [1,n_i]$ and $supp(c_2)\subseteq [i+1,+\infty[$.
The sequence $(n_i)_{i\in\N}$ is called \emph{controllability sequence of $G$}.

\item $G$ is \emph{order controllable} if for every %$c\in G$ and
$i\in \mathbb{N}$ there is $n_i\in \mathbb{N}$ such that for each $c\in G$, there exists
$c_1\in G$ such that $c_{1|[1,i]}=c_{|[1,i]}$, $supp (c_1)\subseteq [1,n_i]$, and order$(c_1)$
divides order$(c_{|[1,n_i]})$
(again, we assume that $n_i$ is the least natural number satisfying this property).
This property implies the existence of $c_2\in G$ such that $c=c_1+c_2$,
$supp(c_2)\subseteq [i+1,+\infty[$, and order$(c_2)$ divides order$(c)$.
Here, the order of $c$ is taken in the usual sense, considering $c$ as an element of the group $G$.
The sequence $(n_i)_{i\in\N}$ is called \emph{order controllability sequence of $G$}.
\end{enumerate}
\end{defn}

\brem
\begin{enumerate}[(i)]
\item Every controllable group is weakly controllable and, if the groups $G_i$ are finite,
then the notions of controllability and weakly controllability are equivalent
(see \cite[Corollary 2.3]{FHS:2017}, where the term uniformly controllable subgroup is used instead of controllable subgroup
that we have adopted here).

\item If $\{G_i : i\in \mathbb{N} \}$ is a family of  finite, abelian, groups and
$G$ is an infinite subgroup of  $\prod\limits_{i\in \mathbb{N}} G_i$ that contains an order controllable
dense subgroup $H$, then $G$ is order controllable as well. (To see this, take an arbitrary element
$z\in G$ and let $[1,m]$ be an arbitrary finite block.
By the density of $H$ in $G$, there is an element $h\in H$ such that $\pi_{[1,n_m]}(z)=\pi_{[1,n_m]}(h)$, where $(n_i)$ denotes the
order controllability sequence of $H$. Now, applying that $H$ is order controllable,
there is $h_1\in H$ such that $\pi_{[1,m]}(h_1)=\pi_{[1,m]}(h)=\pi_{[1,m]}(z)$, $\hbox{supp} (h_1)\subseteq [1,n_m]$
and order$(h_1)$ divides order$(h_{|[1,n_m]})=$order$(z_{|[1,n_m]}$)).
\end{enumerate}
\erem
%\section{Basic facts}

\section{Profinite abelian $p$-groups}

In this section, we describe the structure of profinite abelian $p$-groups.
%Remark that, for $p$-groups, the notion of order controllability can be reduced
%as follows: A profinite abelian $p$-group $G$ is order controllable if for every %$c\in G$ and
%$i\in \mathbb{N}$ there is $n_i\in \mathbb{N}$ such that for each $c\in G$, there exists
%$c_1\in G$ such that $c_{1|[1,i]}=c_{|[1,i]}$, $supp (c_1)\subseteq [1,n_i]$,
%and order$(c_1)\leq$ order$(c_{|[1,n_i]})$
%(we assume that $n_i$ is the least natural number satisfying this property).
%Again, this property implies the existence of $c_2\in G$ such that $c=c_1+c_2$,
%$supp(c_2)\subseteq [i+1,+\infty[$, and order$(c_2)\leq $ order$(c)$.

\begin{lem}\label{lem altura}
Let $\{G_i : i\in \mathbb{N} \}$ be a family of  finite, abelian, $p$-groups and let $G$ be an infinite subgroup of
$\prod\limits_{i\in \mathbb{N}} G_i$ which is order controllable. If $x\in G_{[1,n_i]}[p]$ and $\pi_{[1,i]}(x)\neq 0$,
where $(n_i)_{i\in\N}$ is the order controllability sequence of $G$,
then there exists $\widetilde{x}\in G_{[1,n_i]}[p]$ such that $\pi_{[1,i]}(\widetilde{x})=\pi_{[1,i]}(x)$
and $h(x,G)=h(\widetilde{x},G_{[1,n_i]})$. In the particular case that $\pi_{[1,i-1]}(x)=0$ and there is $j$ such that $n_j<i$
we can take $\widetilde{x}$ such that $h(\widetilde{x},G_{[j+1,n_i]})=h(x,G)=h(x,G_{[j+1,+\infty[})$.
In either case, we take $\widetilde{x}$ with the maximum possible height among those elements satisfying these properties.
\end{lem}

\begin{proof}
Take an element $x\in G_{[1,n_i]}[p]$ with $\pi_{[1,i]}(x)\neq 0$.
Since every group $G_i$ in the product is finite and $x$ has finite support, it follows that $x$ has finite height.
Pick an arbitrary element $y\in G$ such that $x=p^{h}y$ (where $h=h(x,G)$ is the maximal height), which implies
that order$(y)=p^{h+1}$. Since $G$ is order controllable, $y=\widetilde{y}+w$ where $\widetilde{y}\in G_{[1,n_i]}$,
order$(\widetilde{y})= p^{h+1}$, $w\in G_{[i+1,+\infty[}$,  order$(w)\leq p^{h+1}$ and $p^hw(j)=0$ for all $j>n_i$.
Observe that $p^hw\in G_{[i+1,n_i]}[p]$, $\widetilde{x}:=p^h\widetilde{y}\in G_{[1,n_i]}[p]$,
$0\neq \pi_{[1,i]}(x)=\pi_{[1,i]}(\widetilde{x})$, and $h(\widetilde{x},G_{[1,n_i]})=h(x,G)$.

Suppose now that  $\pi_{[1,i-1]}(x)=0$ and there is $j$ such that $n_j<i$. Then $\pi_{[1,n_j]}(x)=\pi_{[1,n_j]}(\widetilde{x})=0$ and
order$(\widetilde{y}_{|[1,n_j]})\leq p^h$. Moreover, $\widetilde{y}=w_1+w_2$, $w_1\in G_{[1,n_j]}$, $w_2\in G_{[j+1,n_i]}$ and
order$(w_1)\leq$ order$(\widetilde{y}_{|[1,n_j]})\leq p^h$. Then $0\neq \widetilde{x}=p^h(w_1+w_2)=p^hw_1+p^hw_2=p^hw_2$,
$\pi_{[1,j]}(w_2)=0$ and order$(w_2)=p^{h+1}$. As a consequence, $h(x,G)=h(\widetilde{x},G_{[1,n_i]})=h(\widetilde{x},G_{[j+1,n_i]})$.
The same argument shows that $h(x,G)=h(x,G_{[j+1,+\infty[})$.
\end{proof}
\mkp

Next follows the main result of this section. It provides sufficient conditions for a subgroup $G$ to be weakly rectangular or an implicit direct product.

\begin{thm}\label{th_products_p}
Let $\{G_i : i\in \mathbb{N} \}$ be a family of  finite, abelian, $p$-groups.  %and let $G=\prod\limits_{i\in \mathbb{N}} G_i$.
If $G$ is an (infinite) order controllable, closed, subgroup of  $\prod\limits_{i\in \mathbb{N}} G_i$ then the following
assertions hold true:
\begin{enumerate}[(i)]
\item There is a generating set $\{y_m: m\in \mathbb{N} \}\subseteq G\cap\bigoplus\limits_{i\in \mathbb{N}} G_i$ for $G$.
\item If $G$ has finite exponent, then there is an isomorphic encoder $$\Phi\colon \prod\limits_{m\in \mathbb{N}} \langle y_m\rangle\to G.$$
As a consequence $G$ is weakly rectangular.
\item Let $p^{h_m+1}$ be the order of $y_m$. If the group
$\sum\limits_{m\in \mathbb{N}} \langle p^{h_m}y_m\rangle$ is weakly observable,
%$G[p]\cap\bigoplus\limits_{i\in \mathbb{N}} G_i$  is isomorphic to $\bigoplus\limits_{i\in \mathbb{N}} \langle y_m\rangle[p]$,
then $G$ is an implicit direct product.
\end{enumerate}
\end{thm}

\begin{proof}
The proof relies on the existence of two increasing sequences of na\-tu\-ral numbers
$(d_k)_{k\geq 1}$ and $(m(k))_{k\geq 0}$,
where $m(0)=0$,
%two sequences of elements in $\mathcal C$, $(x_k)\subseteq \mathcal C[p]\bigcap\bigoplus\limits_{i\in \N}G_i$ and $(y_k)\subseteq \bigoplus\limits_{i\in \N}G_i$
and a sequence of finite subsets $B_k:=\{x_{m(k-1)+1}, \cdots , x_{m(k)}\}\subseteq G[p]\bigcap\bigoplus\limits_{i\in \N}G_i$ satisfying the following conditions:

\begin{enumerate}[(a)]
\item $\pi_{[d_{k-1}+1,{d_k}]}(B_k)$ consists of linearly independent \emph{vectors} in $\pi_{[d_{k-1}+1,{d_k}]}(G[p])$;
\item $\pi_{[d_{k-1}+1,{d_k}]}(B_1\cup\cdots B_k)$ is a generating set of $\pi_{[d_{k-1}+1,{d_k}]}(G[p])$;
\item $\pi_{[1,d_k]}(B_1\cup\cdots B_k)$ forms a basis of $\pi_{[1,d_k]}(G[p])$;
\item if $m(k-1)+1\leq j\leq m(k)$, then $x_j\in G_{[d_{k-1}+1,n_{d_k}]}[p]\setminus\langle x_1,\cdots x_{j-1} \rangle$
and $x_j$ has maximal height $h_j$ in $G$; %or, equivalently, in $G_{[1,n_{d_k}]}$;
%\item [(5)] $m(1)+1\leq j\leq m(2)$.
\item for each $x_j\in B_k$ there is %a nonnegative integer $h_j$ and
an element $y_j\in G_{[1,n_{d_k}]}$
such that $x_j=p^{h_j}y_j$. Furthermore $y_j(i)=0$ for all $j>m(n_i)$;
\item $G[p]=\langle B_1\rangle\bigoplus\cdots\langle B_k\rangle\bigoplus G_{[d_k+1,+\infty[}[p]$
(here, with some notational abuse, we mean {\em vector space direct sum}).

\end{enumerate}
\medskip

Remark that (f) yields
\begin{equation}\label{equ1}
\pi_{[1,d_k]}(G[p])=\pi_{[1,d_k]}(\langle B_1\rangle\bigoplus\cdots\langle B_k\rangle)\ (\forall k\in\N).
\end{equation}
As a consequence, we obtain
$$G[p]\subseteq \overline{\bigoplus\limits_{k=1}^\infty \langle B_k\rangle}\cong\overline{\bigoplus\limits_{m=1}^\infty \langle x_m\rangle}.$$

\medskip

We proceed by induction in order to prove the existence of the sequences $(d_k)_{k\in\N}$, $(m(k))_{k\in\N}$, and $B_k:=\{x_{m(k-1)+1}, \cdots , x_{m(k)}\}$.

Since $G$ is order controllable, there is an order controllability sequence $(n_i)_{i\geq 1}\subseteq \N$ such that
 $\pi_{[1,i]}(G)=\pi_{[1,i]}(G_{[1,n_i]})$ for all $i\in\mathbb{N}$. We have further

$G=G_{[1,n_1]}+G_{[2,+\infty[}=G_{[1,n_1]}+\cdots G_{[i,n_i]}+ G_{[i+1,+\infty[}$,

$G[p]= G_{[1,n_1]}[p]+ G_{[2,+\infty[}[p]= G_{[1,n_1]}[p]+\cdots G_{[i,n_i]}[p]+ G_{[i+1,+\infty[}[p]$.

Remark that, since every group in the
product $G_i$ is finite, all the elements in $(G\bigcap\bigoplus\limits_{i\in \N}G_i)[p]$ have finite height.
%(indeed, if $x\in \mathcal{C}\bigcap\bigoplus\limits_{i\in
%\N}G_i$ had infinite height, then every non-null coordinate $x(i)$ of $x$ would also have
%infinite height, which is impossible because $G_i$ is finite).
\medskip

Let $d_1\in\N$ be the minimum element such that
$$m(1):=\dim\pi_{[1,{d_1}]}(G[p])= \dim\pi_{[1,{d_1}]}(G_{[1,n_{d_1}]}[p])\neq 0.$$
We select an element $x_1\in G_{[1,n_{d_1}]}[p]$ such that $\pi_{[1,{d_1}]}(x_1)\neq\{0\}$ and has maximal height
$h_1:=h(x_1,G)=h(x_1,G_{[1,n_{d_1}]})$, by Lemma \ref{lem altura}. If $\dim  \pi_{[1,{d_1}]}(G_{[1,n_{d_1}]}[p])\neq 1$,
we repeat the same argument in order to obtain an element $x_2\in G_{[1,n_{d_1}]}[p]$ satisfying:
(i) $\pi_{[1,{d_1}]}(x_2)\notin \langle \pi_{[1,{d_1}]}(x_1)\rangle$; and
(ii) $h_1\geq h_2:=h(x_2,G)=h(x_2,G_{[1,n_{d_1}]})$. Furthermore, we select $x_2$ in such a way that has maximal height
among the elements in $G_{[1,n_{d_1}]}$ satisfying (i) and (ii). We go on with this procedure obtaining
a finite subset $B_1=\{x_1, x_2,\cdots x_{m(1)}\}$ such that $\pi_{[1,{d_1}]}(B_1)$ is a basis of
$\pi_{[1,{d_1}]}(G[p])$ and $h_1\geq h_2\geq \cdots\geq h_{m(1)}$, where $h_j=h(x_j, G)=h(x_j, G_{[1,n_{d_1}]})$ is the maximal possible height, $1\leq j\leq m(1)$.
Moreover, associated to every $x_j\in B_1$ there is $y_j\in G_{[1,n_{d_1}]}$
such that $x_j=p^{h_j}y_j$.
Thus the properties (a),\dots, (e) stated above are satisfied for $n=1$.

We now verify (f), that is $$G[p]=\langle B_1\rangle\bigoplus G_{[d_1+1,+\infty[}[p].$$

Indeed, let $0\not=c\in G[p]$. If
 $\pi_{[1,{d_1}]}(c)=0$ then $c\notin \langle B_1\rangle$ since, other\-wise,
we would have $$c=\lambda_1x_1+\cdots+\lambda_{m(1)}x_{m(1)}$$
and $$0=\pi_{[1,{d_1}]}(c)=\lambda_1\pi_{[1,{d_1}]}(x_1)+\cdots~\lambda_{m(1)}\pi_{[1,{d_1}]}(x_{m(1)}),$$
which yields $\lambda_1=\cdots =\lambda_{m(1)}=0$ because  $\pi_{[1,{d_1}]}(B_1)$ is an independent set.

On the other hand, if $\pi_{[1,{d_1}]}(c)\neq0$, then
 $\pi_{[1,{d_1}]}(c)=\pi_{[1,{d_1}]}(b)$ for some $b\in\langle B_1\rangle$. Hence $c=b+w$, and $w=c-b\in G_{[d_1+1,+\infty[}[p]$.
%Observe that $\sC[p]=\langle B_1\rangle\oplusG_{[2,+\infty[}$. Moreover, $\pi_{[1,2]}(\sC[p])=\pi_{[1,2]}(\langle B_1\rangle\oplusG_{[2,+\infty[}[p])=\pi_{[1,2]}(\langle B_1\rangle\oplusG_{[2,n_2]}[p])$.

Now, the inductive procedure for the proof of $n\Rightarrow n+1$ is straightforward. %However, for the sake of simplicity,
We will only sketch the case $n=2$, as it explains well the general case.
%The properties $(1),\dots, (5)$ hold for every $n\in\N$.
%However,

First, since $G$ is infinite,  for some $d_2\in \N$ (take the smallest possible one), we have
$$m(2):=\dim\pi_{[1,d_2]}(G[p])\neq \dim\pi_{[1,d_2]}(\langle B_1\rangle).$$
Furthermore, since $G$ is order controllable, it follows
$$\pi_{[1,d_2]}(G[p])=\pi_{[1,d_2]}( \langle B_1\rangle\bigoplus G_{[d_1+1,+\infty[}[p])=
\pi_{[1,d_2]}( \langle B_1\rangle\bigoplus G_{[d_1+1,n_{d_2}]}[p]).$$
%Indeed, since $G$ is order controllable,we have $c=b+w=b+w_1+w_2$, where $b\in\langle B_1\rangle$, $w\inG_{[d_1+1,+\infty[}[p]$, $w_1\inG_{[d_1+1,n_{d_2}]}[p]$ and $w_2\inG_{[d_2+1,+\infty[}[p]$.

Now, we proceed as in the case $n=1$ in order to obtain a subset $$B_2=\{x_{m(1)+1},\cdots x_{m(2)}\}\subseteq G_{[d_1+1,n_{d_2}]}[p]$$
satisfying the assertions (a),...,(d) and (f) stated above. On the other hand, assertion (e) follows from Lemma \ref{lem altura}.
This completes the inductive argument.
\medskip

Next, we prove the following %, it will suffice to prove that

\textbf{CLAIM:} $$G\bigcap(\bigoplus G_i)\subseteq\overline{\sum\limits_{m\geq 1}\langle y_m\rangle}=G.$$\medskip

%\begin{proof}
\emph{Proof of the Claim:}

First, remark that for each $x\in B_k$, we have $\hbox{order}(x)=p$, $supp(x)\subseteq [d_{k-1}+1,n_{d_k}]$, and $\pi_{[d_{k-1}+1,{d_k}]}(x)\neq 0$.
Furthermore, for each $x\in B_k$,
there exists $y\in G_{[1,n_{d_k}]}$ with $x=p^hy$, $\hbox{order}(y)=p^{h+1}$, where $h=h(x,G)=h(x,G_{[1,n_{d_k}]})$, and such that if $n_j<d_k$, for some $j$,
then $\pi_{[1,j]}(y)=0$ by Lemma \ref{lem altura}.

Set $$Y:=\sum\limits_{m\geq 1}\langle y_m\rangle.$$

We first prove that every element in $\sum\limits_{m\geq 1}\langle x_m\rangle$ has the same height in the group $G$
as in the subgroup $Y\subseteq G$.

Indeed, let $z$ be an arbitrary element in $\sum\limits_{m\geq 1}\langle x_m\rangle$. Then there is some index $k\in\N$ such that

$$z\in\langle B_1\cup\cdots B_k\rangle=\sum\limits_{1\leq m\leq m(k)}\langle x_m\rangle.$$

Set $$Y_k:=\sum\limits_{1\leq m\leq m(k)}\langle y_m\rangle,$$
since $Y_k\subseteq Y\subseteq G$, it is enough to verify that $z$
has the same height in the group $G$ (equivalently, in the subgroup $G_{[1,n_{d_k}]}$) as in the subgroup $Y_k$.

Assume for the moment that
\begin{equation}\label{equ2}
0\neq z=\lambda_{m(k-1)+1}x_{m(k-1)+1}+\cdots\lambda_{r}x_r\in \langle B_k\rangle,
\end{equation}
$0\leq\lambda_j<p$, $\lambda_r\neq 0$, $m(k-1)<j\leq r\leq m(k)$,
where the terms appearing in (\ref{equ2}) are displayed with decreasing height,
that is, in the same order as they are listed in $B_k$.
Thus $$h_j=h(x_j,G)\geq h(x_{j+1},G)=h_{j+1},$$ $m(k-1)<j\leq r\leq m(k).$ %, which means that $h(x_r,G)=h_r=h$.
We also have $\pi_{[1,d_{k-1}]}(z)=0$ and $\pi_{[d_{k-1}+1,d_k]}(z)\neq 0$.

Set $$H_k:=\sum\limits_{m(k-1)< m\leq m(k)}\langle y_m\rangle.$$
%\mkp

Remark that, since the elements $x_j\in B_k$ are taken with decreasing height, it follows that
each $\lambda_jx_j\neq 0$ has the same height in $G$ as in $H_k$. Furthermore, the height of $z$ in $G$ is
\begin{equation}\label{equ3}
h:=h(z,G)=h(x_r,G)=h_r=\min\{h_j: \lambda_j\neq0, m(k-1)< j\leq r\}=h(z,H_k).
\end{equation}
Indeed, if we had $h>h_r$, then we would have selected $z$ (or another vector of the same height) in place of $x_r$
when defining $B_k$. Thus $h(z,G)=h(z,H_k)\leq h(z,Y_k)\leq h(z,G)$, and we are done in this case.

The general case is proved by induction. Assume that whenever $$0\neq z\in\langle B_i\cup\cdots \cup B_k\rangle,$$
where $i$ is the first index such that $\pi_{[1,d_i]}(z)\neq 0$, we have that $h(z,G)=h(z,Y_k)$.
\skp

Reasoning by induction, take an arbitrary element $0\neq z\in\langle B_{i-1}\cup\cdots B_k\rangle$,
where $i-1$ is the first index such that $\pi_{[1,d_{i-1}]}(z)\neq 0$.

Then $z=z_{i-1}+z_i+\cdots z_k$, $z_j\in\langle B_j\rangle$, $i-1\leq j\leq k$, where
$$\pi_{[1,d_{i-1}]}(z_{i-1})=\pi_{[1,d_{i-1}]}(z)$$ and,
from the argument in the paragraph above, the height of $z_j$ in $H_j$ is the same as in $G$, $i-1\leq j\leq k$.

If $h(z_{i-1},G)<h(z_{i}+\cdots z_k,G)$, then
$$h(z,Y_k)\leq h(z,G)=h(z_{i-1},G)=h(z_{i-1},H_{i-1}) \leq h(z_{i-1},Y_{k})\leq h(z_{i-1},G)$$ by (\ref{equ3}).
On the other hand, by the inductive hypothesis, we have
$$h(z_{i}+\cdots z_k,G)=h(z_{i}+\cdots z_k,Y_k).$$

Hence $$h(z_{i-1},Y_k) = h(z_{i-1},G)<h(z_{i}+\cdots z_k,G)=h(z_{i}+\cdots z_k,Y_k),$$
which yields $$h(z,Y_k)= h(z_{i-1},Y_k)=h(z_{i-1},G)=h(z,G).$$
This completes the proof when $h(z_{i-1},G)<h(z_{i}+\cdots z_k,G)$. The case
$h(z_{i}+\cdots z_k,G)<h(z_{i-1},G)$ is analogous.

Therefore, we may assume
without loss of generality that $$h(z_{i-1},G)=h=h(z_{i}+\cdots z_k,G).$$ Moreover, by the inductive hypothesis, we also have
$$h(z_{i-1},Y_k)=h(z_{i-1},G)=h=h(z_{i}+\cdots z_k,G)=h(z_{i}+\cdots z_k,Y_k).$$

Reasoning by contradiction, suppose that $$h(z,G)=r>h(z,Y_k)\geq h.$$ Since $G$ is order controllable we can decompose
$$z=p^{r}y=p^{r}v_{i-1}+p^rw_{i-1},$$
where $$y\in G_{[1,n_{d_k}]},\ v_{i-1}\in G_{[1,n_{d_{i-1}}]},\ w_{i-1}\in G_{[d_{i-1}+1,n_{d_k}]},$$
$$\hbox{order}(p^ry)=\hbox{order}(p^rv_{i-1})=p,$$ $$\pi_{[1,d_{i-2}]}(p^ry)=\pi_{[1,d_{i-2}]}(p^rv_{i-1})=0,$$ and
$$\pi_{[d_{i-2}+1,d_{i-1}]}(p^ry)=\pi_{[d_{i-2}+1,d_{i-1}]}(p^rv_{i-1})=\pi_{[d_{i-2}+1,d_{i-1}]}(z)=\pi_{[d_{i-2}+1,d_{i-1}]}(z_{i-1})\neq 0.$$
Let $\lambda_lx_l$ be the last term in the sum of $z_{i-1}$, then the height of $x_l$ in $G$ coincides with
the height of $z_{i-1}$ in $G$, which is $h$ by (\ref{equ3}). Furthermore,
$p^rv_{i-1}\in G_{[d_{i-2}+1,n_{d_{i-1}}]}[p]$ and $$\pi_{[d_{i-2}+1,d_{i-1}]}(p^rv_{i-1})\notin\pi_{[d_{i-2}+1,d_{i-1}]}(\langle x_{m(i-2)+1},\cdots x_{l-1}\rangle).$$
This is a contradiction with the previous choice of $x_l$ because the height of $p^rv_{i-1}$ in $G$ is $r>h$ and $x_l$ was selected with maximal possible height in $G$.
Therefore, we have proved $h(z,G)=h=h(z,Y_k)=h(z,Y)$.

We now prove that for every $z\in \hbox{tor}(G)$ (the torsion subgroup of $G$) %$z\in G\bigcap(\bigoplus G_i)$, %and for every $k\in\N$,
there is a sequence $$(\lambda_m)\in\prod\limits_{m\in\N}\Z(p^{h_m+1})$$ %where  $\lambda_m\in\Z(p^{h_m+1})$, $1\leq m\leq m(k)$,
such that
\begin{equation}\label{equ4}
z=\lim\limits_{k\rightarrow\infty} \sum\limits_{m=1}^{m(k)} \lambda_my_m
\end{equation}
which is tantamount to
$$\pi_{[1,d_k]}(z)=\pi_{[1,d_k]}\left(\sum\limits_{m=1}^{m(k)} \lambda_my_m\right)$$
for every $k\in\N$.

We proceed by induction on the order $p^s$ of $z$.

Take any element $z\in G[p]$. %$G\bigcap(\bigoplus G_i)[p]$.
By Equation (\ref{equ1}), we know that
$$\pi_{[1,d_k]}(G[p])=\pi_{[1,d_k]}(\langle B_1\rangle\bigoplus\cdots\langle B_k\rangle)$$
%$$G[p]=(\bigoplus\limits_{i=1}^k\langle B_i\rangle)\bigoplus G_{[d_k+1,+\infty[}[p]$$
holds for all $k\in\N.$
Therefore, there is a sequence $$(\alpha_m)\in\prod\limits_{m=1}^\infty \Z(p)$$ such that
$$\pi_{[1,d_k]}(z)=\pi_{[1,d_k]}\left(\sum\limits_{m=1}^{m(k)} \alpha_mx_m\right)=\pi_{[1,d_k]}\left(\sum\limits_{m=1}^{m(k)} \alpha_mp^{h_m}y_m\right)$$
for all $k\in\N.$
This means
$$z=\lim\limits_{k\rightarrow\infty} \sum\limits_{m=1}^{m(k)} \alpha_mp^{h_m}y_m.$$
This completes the proof for $s=1$ if we set $\lambda_m:=\alpha_mp^{h_m}$ for all $m\in\N$.

Now, suppose that the assertion is true when order$(z)\leq p^s$ and pick an arbitrary element
%Reasoning by contradiction, suppose that there is an element
$z\in G$ %$G\bigcap(\bigoplus G_i)$
with order$(z)=p^{s+1}$. %such that  (\ref{equ4}).
%We may assume, without loss of generality, that $z$ has been selected satisfying that $z$ has the smallest possible order %$\pi_{[1,d_1]}(z)$ has the smallest possible
%among the elements with this property. Let $p^s$ be the order of $z$ (notice that $s\geq 2$ by (\ref{equ1})).
Then $p^{s}z$ has order $p$ and therefore belongs to $G[p]$. %$(G\bigcap(\bigoplus G_i))[p]$.
By Equation (\ref{equ1}) again, we know that
there is a sequence
$$(\alpha_m)\in\prod\limits_{m=1}^\infty \Z(p)$$ such that
$$\pi_{[1,d_k]}(p^sz)=\pi_{[1,d_k]}\left(\sum\limits_{m=1}^{m(k)} \alpha_mx_m\right)=\pi_{[1,d_k]}\left(\sum\limits_{m=1}^{m(k)} \alpha_mp^{h_m}y_m\right)$$
for all $k\in\N$.
%$$\pi_{[1,d_1]}(p^{s-1}z)=\pi_{[1,d_1]}\left(\sum\limits_{m=1}^{m(1)} \alpha_mx_m\right)=\pi_{[1,d_1]}\left(\sum\limits_{m=1}^{m(1)} \alpha_mp^{h_m}y_m\right).$$

Now, since we have chosen each element $x_m$ with the maximal possible height, it follows that $s\leq h_m$ for all $1\leq m\leq m(k)$,
and $k\in\N$. Therefore
$$\pi_{[1,d_k]}(p^{s}z)=\pi_{[1,d_k]}\left(p^{s}\sum\limits_{m=1}^{m(k)} \alpha_mp^{h_m-s}y_m\right),$$
for all $k\in\N$, which yields
$$\pi_{[1,d_k]}\left(p^{s}(z -\sum\limits_{m=1}^{m(k)} \alpha_mp^{h_m-s}y_m)\right)=0$$
for all $k\in\N$.

Set $$v=\lim\limits_{k\rightarrow\infty}\sum\limits_{m=1}^{m(k)} \alpha_mp^{h_m-s}y_m\in G$$
where the limit exists, and therefore $v$ is well defined, because $y_m(i)=0$ for all $m>m(n_i)$.
Then we have $z=v+(z-v)$, where order$(z-v)\leq p^s$. By the inductive hypothesis, there is a sequence $$(\mu_m)\in\prod_{m\in\N}\Z(p^{h_m+1})$$ such that
$$z-v=\lim\limits_{k\rightarrow\infty} \sum\limits_{m=1}^{m(k)} \mu_my_m.$$
Therefore
$$z=\lim\limits_{k\rightarrow\infty} \sum\limits_{m=1}^{m(k)} (\alpha_mp^{h_m-s}+\mu_m)y_m.$$
This completes the proof of the inductive argument. Therefore, it is proved that

$$G\bigcap(\bigoplus G_i)\subseteq \hbox{tor}(G)\subseteq\overline{\sum\limits_{m\geq 1}\langle y_m\rangle}.$$
Since $G$ is closed and order controllable, it follows
$$\overline{G\bigcap(\bigoplus G_i)}=G=\overline{\sum\limits_{m\geq 1}\langle y_m\rangle}.$$
This completes the proof of the Claim.
%\end{proof}
\mkp

We now proceed with the proof of the three assertions formulated in this theorem.\mkp

\noindent (i) We will now prove that $G$ is topologically %isomorphic to $\prod\langle y_m\rangle$ and, thereby,
%that $G$ is a weakly rectangular subgroup
generated by the set $\{y_m : m\in\N\}$.

First,
observe that the finite subgroup $\langle y_m\rangle$, generated by $y_m$ in $G$,
is isomorphic to $\mathbb{Z}(p^{h_m+1})$ for every $m\geq1$.
Thus, without loss of generality, we may replace the group $\langle y_m\rangle$ by $\mathbb{Z}(p^{h_m+1})$ in the sequel.
Consider now the group $\prod\limits_{m\in \N} \mathbb{Z}(p^{h_m+1})$, equipped
with the product topology and its dense subgroup $\bigoplus\limits_{m\in \N} \mathbb{Z}(p^{h_m+1})$.
%By the foregoing Claim,
Set
$$\Phi: \bigoplus\limits_{m\in\N} \mathbb{Z}(p^{h_m+1})\longrightarrow G\bigcap(\bigoplus G_i)\leq G$$
defined by
$$\Phi[(k_1,\dots,k_m,\dots)]= \sum\limits_{m=1}^\infty k_my_m.$$ Since only
finitely many $k_m$ are non-null, the map $\Phi$ is clearly well defined.
We will prove that $\Phi$ is also a topological group isomorphism on its image.

In order to verify that $\Phi$ is one-to-one, suppose there is a sequence
$$(k_1,\cdots,k_r,0,\cdots)\in\ker f,\ 0\leq k_j<p^{h_j+1},$$  with some $k_j\neq 0$. Then we have

$$k_1y_1+\cdots+k_ry_r=0.$$

Expressing every $k_j\neq0$ in base $p$, we obtain $k_j=a_{h_j}^{(j)}p^{h_j}+\cdots+a_1^{(j)}p+a_0^{(j)}$,
$0\leq a_i^{(j)}<p$, $0\leq i\leq h_j$, $1\leq j\leq r$. Let $p^{s_j}$ the minimal power of $p$ that
appears in the expression of $k_j\neq 0$. Since $y_j$ has order $p^{h_j+1}$ the order of $k_jy_j$ is $p^{h_j-s_j+1}$.

Defining $d:=\max\{h_j-s_j: k_j\neq0, 1\leq j\leq r\}$ and multiplying by $p^d$ the equality above, we obtain
an expression as follows
$$p^d\left((a_{h_{i_1}}^{(i_1)}p^{h_{i_1}}+\cdots+a_{h_{i_1-d}}^{(i_1)}p^{h_{i_1}-d})y_{i_1}+\cdots
(a_{h_{i_l}}^{(i_l)}p^{h_{i_l}}+\cdots+a_{h_{i_l-s}}^{(i_l)}p^{h_{i_l}-d})y_{i_l}\right)=0,$$
where we have only considered those elements $\{y_{i_j}\}_{j=1}^l$ such that
$h_{i_1}-s_{i_1}=\cdots h_{i_l}-s_{i_l}=d$.  Since $p^{h_{i_j}}y_{i_j}=x_{i_j}$ has order $p$,
we have $$a_{s_{i_1}}^{(i_1)}x_{i_1}+\cdots a_{s_{i_l}}^{(i_l)}x_{i_l}=0.$$
Since the elements $\{x_{i_1},\cdots,x_{i_l}\}$ are all independents, it follows that
$$a_{s_{i_1}}^{(i_1)}=\cdots a_{s_{i_l}}^{(i_l)}=0.$$ This is a contradiction which completes the proof.
Therefore $\Phi$ is 1-to-1.

The sequence $(y_m)$ that we have defined above verifies that $y_m(i)=0$ for all $m> m(n_{i})$. As a consequence,
we have that $\lim\limits_{m\rightarrow \infty} y_m(i)=0$
for all $i\in \N$, which implies the continuity of $\Phi$. Indeed,
let $(z_\alpha)$ be a sequence in $\bigoplus\limits_{m\in\N} \mathbb{Z}(p^{h_m+1})$
converging to $0$. If $V_i=(0,\cdots 0)\times\prod\limits_{j>i}G_j$
is an arbitrary basic \nbd of  $0$ in $\prod\limits_{j\in\mathbb{N}}G_j$, since $(0,\cdots 0)\times\prod\limits_{j>m(n_i)}\mathbb{Z}(p^{h_j+1})$ is a \nbd of $0$
in $\prod\limits_{j\in \mathbb{N}}\mathbb{Z}(p^{h_j+1})$, then there is $\alpha_i$
such that ${z_\alpha}_{|[1,m(n_i)]}=0$ for all $\alpha\geq\alpha_i$. Therefore
$z_{\alpha}=(0,\cdots,0,k_{{m(n_i)+1},\alpha},\cdots)$ and
$\Phi(z_\alpha)=\sum\limits_{m>m(n_i)}k_{m,\alpha}y_m\in V_i$, for all
$\alpha\geq\alpha_i$ and for all $i\in \mathbb{N}$.
Thus, the sequence $(\Phi(z_\alpha))$ converges to $\Phi(0)=0$,
which verifies the continuity of $\Phi$.

As a consequence, there is a continuous extension
$$\Phi: \prod\limits_{m\in\N} \mathbb{Z}(p^{h_m+1})\longrightarrow G$$
that we still denote by $\Phi$ for short, which is continuous and onto. Furthermore, it is easily seen that it holds
$$\Phi[(k_m)]= \sum\limits_{m=1}^\infty k_my_m.$$
Remark that, since $y_m(i)=0$ for all $m> m(n_{i})$, it follows that
$$\sum\limits_{m=1}^\infty k_my_m(i)$$ reduces to a finite sum for all $i\in\N$.
Therefore $\Phi$ is well defined. This proves that $\{ y_m : m\in\N\}$ is a generating set for $G$.
%and thereby we have that $G$ is finitely generated.
\mkp

\noindent (ii) Next we prove that if $G$ has finite exponent then $\Phi$ is $1$-to-$1$ on $\prod\limits_{m\in\N} \mathbb{Z}(p^{h_m+1})$ and,
as a consequence, that $\Phi$ is an isomorphic encoder and $G$ is weakly rectangular.

For that purpose, it will suffice to check that $\ker \Phi =\{0\}$.

We proceed by induction on the order $p^s$ of the elements $\mathbf v:=(\lambda_m)\in\ker \Phi$.

Suppose order$(\vv)=p$, which means $\lambda_m=\alpha_mp^{h_m}$, $0\leq \alpha_m<p$, for all $m\in\N$. We have
$$\Phi(\vv)=\sum\limits_{m=1}^{\infty} \alpha_mp^{h_m}y_m=\sum\limits_{m=1}^{\infty} \alpha_mx_m=0.$$
For every $l\in\N$, set $\mathbf v_l:=(\mu_m)$, where $\mu_m=\lambda_m$ if $1\leq m\leq m(l)$
and $\mu_m=0$ if $m>m(l)$.  It follows that $\lim\limits_{l\rightarrow\infty}\vv_l=\vv$.
By the continuity of $\Phi$ we obtain
$$\lim\limits_{l\rightarrow\infty}\sum\limits_{m=1}^{m(l)}\alpha_mx_m=\sum\limits_{m=1}^{\infty} \alpha_mx_m=0.$$
Thus, for every $k\in\N$, there is $l_k\in\N$ such that
$$\pi_{[1,d_k]}(\Phi(\vv_l))=\pi_{[1,d_k]}\left(\sum\limits_{m=1}^{m(l)}\alpha_mx_m\right)=
\sum\limits_{m=1}^{m(l)}\alpha_m\pi_{[1,d_k]}(x_m)=0$$
for all $l\geq l_k$. On the other hand
$$\Phi(\vv_l)=\sum\limits_{m=1}^{m(1)}\alpha_mx_m+\sum_{\underbrace{m=m(1)+1}_{\pi_{[1,d_1]}(x_m)=0}}^{m(2)}\alpha_mx_m+\cdots
\sum_{\underbrace{m=m(k)+1}_{\pi_{[1,d_k]}(x_m)=0}}^{m(l)}\alpha_mx_m,$$
where $0\leq \alpha_m<p$, then for $l\geq l_k$, we have
$$\pi_{[1,d_1]}(\Phi(\vv_l))=\sum\limits_{m=1}^{m(1)}\alpha_m\pi_{[1,d_1]}(x_m)=0$$ and since
$\pi_{[1,d_1]}(B_1)=\{\pi_{[1,d_1]}(x_1),\cdots, \pi_{[1,d_1]}(x_{m(1))}\}$ is a basis for
$\pi_{[1,d_1]}(G[p])$ we obtain that
$\alpha_1=\cdots=\alpha_{m(1)}=0$.

In like manner, from $$\pi_{[1,d_2]}(\Phi(\vv_l))=\sum\limits_{m=m(1)+1}^{m(2)}\alpha_m\pi_{[1,d_2]}(x_m)=0,$$
we deduce that $\alpha_{m(1)+1}=\cdots=\alpha_{m(2)}=0$. Therefore, iterating this argument, we obtain
$\alpha_{1}=\cdots=\alpha_{m(k)}=0$.
Since $\lambda_m=\alpha_mp^{h_m}$ for all $1\leq m\leq m(l)$, it follows that $\lambda_m=0$ for all $1\leq m\leq m(k)$.
Since this holds for every $k\in\N$, it follows that $\lambda_m=0$ for all $m\in\N$. This completes the proof for $s=1$.

Now, suppose that the assertion is true when order$(\vv)\leq p^s$ and pick an arbitrary element
$\vv=(\lambda_m)\in \ker \Phi$ such that order$(\vv)=p^{s+1}$.
Then $p^{s}\vv\in\ker \Phi$ and has order $p$. Therefore, the arguments above applied to $p^{s}\vv$ yields that $p^s\vv =0$,
which is a contradiction. %In other words, $\vv\in\ker \Phi$ and order$(\vv)=p^s$.
By the inductive assumption, it follows that $\vv =0$, which completes the proof.

Therefore, we have proved that
$$\Phi: \prod\limits_{m\in\N} \mathbb{Z}(p^{h_m+1})\longrightarrow G$$
is $1$-to-$1$. The compactness of the domain implies that $\Phi$ is a topological group isomorphism onto $G$.\mkp

\noindent (iii) Assume that  $\sum\limits_{m\in \mathbb{N}} \langle x_m\rangle$ is weakly observable.
This means
$$\overline{\sum\limits_{m\in \mathbb{N}} \langle x_m\rangle}\cap \bigoplus G_i = \sum\limits_{m\in \mathbb{N}} \langle x_m\rangle.$$

We have to verify that the map $$\Phi : \bigoplus\limits_{m\in\N} \mathbb{Z}(p^{h_m+1})\longrightarrow G\cap \bigoplus G_i$$
is onto. Reasoning by contradiction, suppose there is an element
$$z\in G\cap \bigoplus G_i\setminus \sum\limits_{m\in \mathbb{N}} \langle y_m\rangle,$$
which has the smallest possible order, $p^{s+1}, s\geq 0$, of an
element with this property. Since, by the foregoing Claim, we have that
$$G[p]\cap \bigoplus G_i\subseteq  \overline{\sum\limits_{m\in \mathbb{N}} \langle x_m\rangle},$$
it follows that $$G[p]\cap \bigoplus G_i = \sum\limits_{m\in \mathbb{N}} \langle x_m\rangle.$$
Therefore, since $p^sz\in G[p]\cap \bigoplus G_i$, there must be a finite subset $J\subseteq \N$ such that
$$p^sz=\sum\limits_{m\in J} \ga_mx_m = \sum\limits_{m\in J} \ga_mp^{h_m}y_m,\ 0 < \ga_m <p.$$
We have already proved that $s\leq h_m$ for all $m\in J$, which yields
$$p^sz= p^s\sum\limits_{m\in J} \ga_mp^{h_m-s}y_m.$$
Set $$v:= \sum\limits_{m\in J} \ga_mp^{h_m-s}y_m = f((\ga_mp^{h_m-s})).$$ Then $p^s(z-v)=0$.

If $s=0$, we obtain $z=v$ and we are done.

So, assume that $s>0$. In this case, we have an element $z-v\in G\cap \bigoplus G_i$, whose order is $p^s$.
By our initial assumption, this means that $z-v=f((\gl_m))$ for some element $(\gl_m)\in \bigoplus\limits_{m\in\N} \mathbb{Z}(p^{h_m+1}).$
Therefore $$z=v+f((\gl_m)) = f((\ga_mp^{h_m-s}+\gl_m)),$$ which is a contradiction. This completes the proof.
\end{proof}

\bexm
Let $G\subseteq\prod\limits_{i\geq 1} G_i$, where $G_i=\mathds{Z}(2^{2})$, be the subgroup generated by
the set $\{y_n : n\in\N\}$, where
$y_1\in G_{[1,2]}$ with $y_1(1)=2$ and $y_1(2)=1$, and
$y_n\in G_{[n,n+1]}$ with $y_n(n)=y_n(n+1)=1$ for $n>1$.

The group $G$ is not order controllable. Indeed, for any block $[1,m]$, pick
$y\in G$ such that $y(n)=2$ for all $1\leq n\leq m+1$, which
only admits the sum $y=z_m+z$ with the first part $z_m\in G_{[1,m+1]}$, where
$z_m(n)=y(n)=2$, $1\leq n\leq m$ and $z_m(m+1)=1$, $m\geq 1$. Then $order(y_{|[1,m+1]})=2$ but $order(z_m)=4$.

On the other hand, it is easily seen that  $\overline G^{\prod G_i}$ is an implicit direct product
of the family $\{ G_i : i\in \N\}$. Therefore, the choice of an appropriate generating set is essential
in order to determine whether a subgroup of a product is weakly rectangular or an implicit direct product.
\eexm

\section{Main Result}

%For every group $G$ we denote by $(G)_p$ the largest $p$-subgroup of $G$ and
%$\mathbb{P}_G=\{p\in\mathbb{P}\,:\, G\hbox{ contains a }p-\hbox{subgroup}\}$
%where %$p\in \mathbb{P}_G$ and
%$\mathbb{P}$ is the set of all prime numbers.

Let $G$ be a closed subgroup of $X=\prod\limits_{i\in\mathbb{N}}G_i$ (a countable product of finite abelian groups).
Since each group $G_i$ is finite and abelian, by the fundamental structure theo\-rem of finite abelian groups, we have that
every group $G_i$ is a  finite sum of finite $p$-groups, that is $G_i\cong\bigoplus\limits_{p\in\mathbb{P}_i}(G_i)_p$ and
$\mathbb{P}_i=\mathbb{P}_{G_i}$ is finite, $i\in\mathbb{N}$. Note that $\mathbb{P}_X=\cup\mathbb{P}_i$.
We have $$\prod\limits_{i\in\mathbb{N}}G_i\cong\prod\limits_{i\in\mathbb{N}}(\prod\limits_{p\in\mathbb{P}_i}(G_i)_p)\cong\prod
\limits_{p\in\mathbb{P}_X}(\prod\limits_{i\in\mathbb{N}_p}(G_i)_p)$$
where $\mathbb{N}_p=\{i\in\mathbb{N}\,:\,G_i\hbox{ has a nontrivial}\ p-\hbox{subgroup}\}$.

\noindent Thus $$(X)_p\cong\prod\limits_{i\in\mathbb{N}_p}(G_i)_p.$$

\noindent Consider the embedding
$$j:G\hookrightarrow \prod\limits_{p\in\mathbb{P}_G}(\prod\limits_{i\in\mathbb{N}_p}(G_i)_p)$$ and the canonical projection
$$\pi_p:\prod\limits_{p\in\mathbb{P}_G}(\prod\limits_{i\in\mathbb{N}_p}(G_i)_p)\rightarrow\prod\limits_{i\in\mathbb{N}_p}(G_i)_p.$$

\noindent Set $G^{(p)}=(\pi_p\circ j)(G)$, that is a compact group.
We have
$$(G)_p\cong G^{(p)}.$$

%\medskip
Now, it is easily seen that if $G$ is order controllable then $(G)_p$ has this property
for each $p\in \mathbb{P}_G$. Taking this fact into account, we obtain the following result that answers to Question \ref{Question}
for products of finite abelian groups.\mkp

We can now prove Theorem {\bf A}.

%\begin{thm}\label{th_products} (t$h_{}$products)
%Let $G$ be an order controllable, closed, subgroup of a countable product $X=\prod\limits_{i\in\mathbb{N}}G_i$ of finite abelian groups $G_i$.
%Then the following assertions hold true:
%\begin{enumerate}[(a)]
%\item There is a generating set $\{y_{m}^{(p)}\,:\,m\in\mathbb{N},p\in\mathbb{P}_G\}\subseteq G\cap(\bigoplus\limits_{i\in\mathbb{N}}G_i)$
%for $G$ such that $\{y_m^{(p)}\,:\,m\in\mathbb{N}\}\subseteq(G\cap(\bigoplus\limits_{i\in\mathbb{N}}G_i))_p$.
%\item If $G$ has finite exponent,
%then $G$ is topologically isomorphic to $\prod\limits_{\begin{array}{l}_{m\in\mathbb{N}}\\^{p\in\mathbb{P}_G}\end{array}}\langle y_m^{(p)}\rangle$
%and, as a consequence, weakly rectagular.
%\item If $\bigoplus\limits_{m\in\mathbb{N}} \langle y_m^{(p)}\rangle[p]$ is weakly observable for each $p\in \mathbb{P}_G$,
%then $G$ is an implicit direct product.
%\end{enumerate}
%\end{thm}
\begin{proof}[\textbf{Proof of Theorem \ref{theorem_A}}].
Since $G\cap(\bigoplus\limits_{i\in\mathbb{N}}G_i)$ is dense in $G$,
we have that
$$(\pi_p\circ j)(G\cap(\bigoplus\limits_{i\in\mathbb{N}}G_i))=G^{(p)}\cap\bigoplus\limits_{i\in \mathbb{N}_p}(G_i)_p$$ is dense in $G^{(p)}$.
Thus (a) is a direct consequence of Theorem \ref{th_products_p}. That is, for each $p\in \mathbb{P}_G$, there is a sequence
 $$\{y_m^{(p)}\,:\,m\in\mathbb{N}\}\subseteq G^{(p)}\cap\bigoplus\limits_{i\in \mathbb{N}_p}(G_i)_p$$
 such that $\{y_{m}^{(p)}\,:\,m\in\mathbb{N}\}$ is a generating set for
 $G^{(p)}$. Furthermore, observe that if $p\in\mathbb{P}_G$, then $G^{(p)}\cap\bigoplus\limits_{i\in \mathbb{N}_p}(G_i)_p\cong (G\cap(\bigoplus\limits_{i\in \mathbb{N}}G_i))_p$. Thus, using this isomorphism, we may assume with some notational abuse that
$$\{y_m^{(p)}\,:\,m\in\mathbb{N}\}\subseteq (G\cap(\bigoplus\limits_{i\in \mathbb{N}}G_i))_p$$
 Therefore, the sequence
 $$\{y_m^{(p)}\,:\,m\in\mathbb{N},p\in\mathbb{P}_G,p\in\mathbb{P}_G\}\subseteq G\cap(\bigoplus\limits_{i\in\mathbb{N}}G_i)$$
 is a generating set for $G$.

%(otherwise, it is the neutral element).
In order to prove (b), we apply Theorem \ref{th_products_p} again and, since $G$ has finite exponent,
for each $p\in \mathbb{P}_G$, we have
that $G^{(p)}\cong\prod\limits_{m\in\mathbb{N}}\langle y_m^{(p)}\rangle$, which yields (b).

Finally, If $\bigoplus\limits_{m\in\mathbb{N}} \langle y_m^{(p)}\rangle[p]$ is weakly observable for each $p\in \mathbb{P}_G$,
then $G^{(p)}$ is an implicit direct product for every $p\in \mathbb{P}_G$, which again implies that
$G$ is an implicit direct product.
\end{proof}
%\medskip

\bqtn
%\noindent {\bf Question:}
Under what conditions is it possible to extend Theorem A to non-Abelian groups?
\eqtn
\mkp

\textbf{Acknowledgment:} The authors thank Dmitri Shakahmatov for several helpful comments.

\end{document}